\documentclass{amsart}
\usepackage{xcolor}
\usepackage{hyperref}
\hypersetup{
	colorlinks=true,
	citecolor=blue,
	linkcolor=blue,
	filecolor=magenta,      
	urlcolor=cyan,
}
\usepackage[capitalise,noabbrev]{cleveref}

\usepackage[all,cmtip]{xy}

\usepackage{amsmath,amscd}
\usepackage{amssymb}
\usepackage{mathtools}
\usepackage{stmaryrd}
\usepackage{url}
\usepackage{tikz-cd}
\usepackage{enumitem}

\usepackage{setspace}
\usepackage{microtype}

\usepackage{times}

\usepackage{tikz}

\usepackage{enumitem,kantlipsum}

\address{Department of Algebra, Faculty of Mathematics and Physics, Charles University in Prague, Sokolovsk\'a 83, 186 75 Praha, Czech Republic}

\email{shaul@karlin.mff.cuni.cz}

\newtheorem{thm}[equation]{Theorem}
\newtheorem*{thm*}{Theorem}
\newtheorem*{cor*}{Corollary}
\newtheorem*{dfn*}{Definition}

\newtheorem{cthm}{Theorem}
\newtheorem{ccor}[cthm]{Corollary}

\newtheorem{cor}[equation]{Corollary}
\newtheorem{prop}[equation]{Proposition}
\newtheorem{lem}[equation]{Lemma}

\theoremstyle{definition}

\newtheorem{rem}[equation]{Remark}

\newcommand{\inj}{\hookrightarrow}

\newcommand{\opn}{\operatorname}
\newcommand{\cat}[1]{\operatorname{\mathsf{#1}}}

\newcommand{\mfrak}[1]{\mathfrak{#1}}

\newcommand{\mrm}[1]{\mathrm{#1}}
\newcommand{\mbb}[1]{\mathbb{#1}}

\newcommand{\K}{\mbb{K} \hspace{0.05em}}

\newcommand{\p}{\mfrak{p}}

\newcommand{\injdim}{\operatorname{inj\,dim}}
\newcommand{\projdim}{\operatorname{proj\,dim}}

\def\skewtimes{\ltimes\!}

\title{Categorical properties of reduction functors over non-positive DG-rings}
\author{Liran Shaul}

\begin{document}

\begin{abstract}
Given a non-positive DG-ring $A$,
associated to it are the reduction and coreduction functors $F(-) = \mrm{H}^0(A)\otimes^{\mrm{L}}_A -$ and $G(-) = \mrm{R}\opn{Hom}_A(\mrm{H}^0(A),-)$, 
considered as functors $\cat{D}(A) \to \cat{D}(\mrm{H}^0(A))$,
as well as the forgetful functor $S:\cat{D}(\mrm{H}^0(A)) \to \cat{D}(A)$.
In this paper we carry a systematic study of the categorical properties of these functors.
As an application, a new descent result for vanishing of $\opn{Ext}$ and $\opn{Tor}$ over ordinary commutative noetherian rings is deduced.
\end{abstract}

\setcounter{section}{-1}

\numberwithin{equation}{section}
\maketitle

\section{Introduction}

Let $A$ be a non-positively graded differential graded ring,
where we grade cohomologically. 
In this situation, 
there is a natural map of DG-rings $\pi_A:A\to \mrm{H}^0(A)$,
which give rise to three important functors associated to $A$.
They are the two functors $\cat{D}(A) \to \cat{D}(\mrm{H}^0(A))$ defined as
\[
F(-):=\mrm{H}^0(A)\otimes^{\mrm{L}}_A -,
\quad 
G(-):=\mrm{R}\opn{Hom}_A(\mrm{H}^0(A),-),
\]
as well as the forgetful functor $S:\cat{D}(\mrm{H}^0(A)) \to \cat{D}(A)$.
The functors $F,G$ are called the reduction and coreduction functors respectively,
and are of crucial importance in the study of non-positive DG-rings.
We first learned about these functors from \cite{YeDual}.

The aim of this paper is to perform a systematic study of the categorical properties of these three foundational functors.
Our first result concerns fullness and faithfulness of these functors.
\begin{cthm}
Let $A$ be a non-positive DG-ring.
Then the following holds:
\begin{enumerate}
\item The reduction functor $\mrm{H}^0(A)\otimes^{\mrm{L}}_A -:\cat{D}(A) \to \cat{D}(\mrm{H}^0(A))$ is full if and only if it is faithful if and only if $A \cong \mrm{H}^0(A)$,
so that $A$ is equivalent to the underlying ring $\mrm{H}^0(A)$.
\item The coreduction functor $\mrm{R}\opn{Hom}_A(\mrm{H}^0(A),-):\cat{D}(A) \to \cat{D}(\mrm{H}^0(A))$ is full if and only if it is faithful if and only if $A \cong \mrm{H}^0(A)$,
so that $A$ is equivalent to the underlying ring $\mrm{H}^0(A)$.
\item The forgetful functor $\cat{D}(\mrm{H}^0(A)) \to \cat{D}(A)$ is full if and only if $A \cong \mrm{H}^0(A)$,
so that $A$ is equivalent to the underlying ring $\mrm{H}^0(A)$.
\item There exist examples of DG-rings $A$ which are not equivalent to the underlying ring $\mrm{H}^0(A)$,
such that the forgetful functor $\cat{D}(\mrm{H}^0(A)) \to \cat{D}(A)$ is faithful,
and other examples where it is not faithful. In both cases, examples exist even when $A$ is a commutative noetherian DG-ring with bounded cohomology.
\end{enumerate}
\end{cthm}

The unfortunate fact that the forgetful functor $\cat{D}(\mrm{H}^0(A)) \to \cat{D}(A)$ may fail to be faithful complicates the study of DG-rings.
On the positive side, we prove in \cref{thm:sym-vanish} that while the forgetful functor is not necessarily faithful, it does able to detect vanishing of $\mrm{R}\opn{Hom}$ and derived tensor products, under commutativity and boundedness assumptions.
As a corollary of this, we obtain in \cref{cor:torExt} the following descent result for vanishing of $\opn{Ext}$ and $\opn{Tor}$:
\begin{ccor}
Let $A$ be a commutative noetherian ring,
let $\mathbf{x}=x_1,\dots,x_n$ be a finite sequence of elements in $A$, and let $M,N$ be $A$-modules.
Assume that $\mathbf{x}$ is both an $M$-regular and an $N$-regular sequence.
\begin{enumerate}
\item If $\opn{Tor}_n^A(M,N) = 0$ for all $n\ge 0$,
then $\opn{Tor}_n^{A/\mathbf{x}A}(M/\mathbf{x}M,N/\mathbf{x}N) = 0$ for all $n \ge 0$.
\item If $\opn{Ext}^n_A(M,N) = 0$ for all $n\ge 0$,
then $\opn{Ext}^n_{A/\mathbf{x}A}(M/\mathbf{x}M,N/\mathbf{x}N) = 0$ for all $n \ge 0$.
\end{enumerate}
\end{ccor}

It is particularly interesting to note that in the above result, we do not assume that $\mathbf{x}$ is $A$-regular.
Thus, the passage from $A$ to $A/\mathbf{x}A$ behaves bad from a homological point of view, 
so to prove this completely elementary result, we have to rely on differential graded methods.

Recall, that a functor $T:\mathcal{C}\to \mathcal{D}$ is called conservative if for any morphism $f$ in $\mathcal{C}$,
if $T(f)$ is an isomorphism, 
then $f$ is an isomorphism.
If $\mathcal{C}$ and $\mathcal{D}$ happen to be triangulated categories,
and $T$ is a triangulated functor,
by considering the cone of $f$, 
this is clearly equivalent to the fact that for any object $M$ in $\mathcal{C}$,
if $T(M) \cong 0$ then $M \cong 0$.
Our next main result,
discusses the conservative property for the reduction, coreduction and the forgetful functor.

\begin{cthm}
Let $A$ be a non-positive DG-ring.
\begin{enumerate}
\item The forgetful functor $\cat{D}(\mrm{H}^0(A)) \to \cat{D}(A)$ is always conservative.
\item The restriction of the reduction functor $\mrm{H}^0(A)\otimes^{\mrm{L}}_A -:\cat{D}^{-}(A) \to \cat{D}^{-}(\mrm{H}^0(A))$ to the bounded above derived categories is always conservative.
\item The restriction of the coreduction functor
$\mrm{R}\opn{Hom}_A(\mrm{H}^0(A),-):\cat{D}^{+}(A) \to \cat{D}^{+}(\mrm{H}^0(A))$ to the bounded below derived categories is always conservative.
\item If $A$ has bounded cohomology,
then the reduction functor $\mrm{H}^0(A)\otimes^{\mrm{L}}_A -:\cat{D}(A) \to \cat{D}(\mrm{H}^0(A))$ and the coreduction functor $\mrm{R}\opn{Hom}_A(\mrm{H}^0(A),-):\cat{D}(A) \to \cat{D}(\mrm{H}^0(A))$ are both conservative on the entire unbounded derived categories.
\item There exist a commutative DG-ring $A$ with unbounded cohomology such that the reduction functor $\mrm{H}^0(A)\otimes^{\mrm{L}}_A -:\cat{D}(A) \to \cat{D}(\mrm{H}^0(A))$ and the coreduction functor $\mrm{R}\opn{Hom}_A(\mrm{H}^0(A),-):\cat{D}(A) \to \cat{D}(\mrm{H}^0(A))$ are not conservative.
\end{enumerate}
\end{cthm}

We should remark that most of this result is not new:
item (1) is trivial,
item (2) was shown in \cite[Proposition 3.1]{YeDual},
item (3) is \cite[Proposition 3.4]{ShINJ},
item (4) is \cite[Theorem 4.5]{ShWi}.
Our new contribution here is item (5),
but we chose to present the above theorem as a whole, 
as it gives a complete summary of the answer to these questions.

\section{Preliminaries}

A non-positive DG-ring $A = \bigoplus_{n=-\infty}^0 A^n$ is a graded ring $A$ together with a differential $d:A \to A$ of degree $+1$ which satisfies a Leibniz rule.
See \cite{Keller,YeBook} for background and notation regarding DG-rings and their derived categories.
All DG-rings in this paper are assumed to be non-positive.
A DG-ring $A$ is called commutative if it is graded-commutative and homogeneous elements $a$ of odd degree satisfy $a^2 = 0$.
For a non-positive DG-ring $A$,
we have that $\mrm{H}^0(A)$ is an ordinary ring.
Moreover, if $M$ is a DG-module over $A$,
then for any $n \in \mathbb{Z}$,
it holds that $\mrm{H}^n(M)$ is a $\mrm{H}^0(A)$-module.
We say that $A$ has bounded cohomology if $\mrm{H}^n(A) = 0$ for all $n\ll 0$.

As noted above, associated to $A$ are the reduction functor 
\[
F(-):=\mrm{H}^0(A)\otimes^{\mrm{L}}_A -:\cat{D}(A) \to \cat{D}(\mrm{H}^0(A)),
\]
the coreduction functor
\[
G(-):=\mrm{R}\opn{Hom}_A(\mrm{H}^0(A),-) : \cat{D}(A) \to \cat{D}(\mrm{H}^0(A)),
\]
and the forgetful functor $S:\cat{D}(\mrm{H}^0(A)) \to \cat{D}(A)$.
This give rise to an adjoint triple $F \dashv S \dashv G$.
In other words, the functor $F$ is a left adjoint to $S$,
while the functor $G$ is right adjoint to $S$.

The fact that $F, G$ and $S$ have adjoints allows one to reduce questions about their categorical properties to questions about the unit and counit maps of the adjunction:

\begin{prop}\label{prop:adjfaith}
Let $L \dashv R$ be an adjoint pair of functors,
where $L:\mathcal{C} \to \mathcal{D}$,
and its right adjoint $R:\mathcal{D} \to \mathcal{C}$.
\begin{enumerate}
\item The functor $R$ is faithful if and only if for any object $M$ of $\mathcal{D}$, the counit map $L(R(M)) \to M$ is an epimorphism.
\item The functor $R$ is full if and only if for any object $M$ of $\mathcal{D}$, the counit map $L(R(M)) \to M$ is a split monomorphism.
\item The functor $L$ is faithful if and only if for any object $M$ of $\mathcal{C}$, the unit map $M \to R(L(M))$ is a monomorphism.
\item The functor $L$ is full if and only if for any object $M$ of $\mathcal{C}$, the unit map $M \to R(L(M))$ is a split epimorphism.
\end{enumerate} 
\end{prop}
\begin{proof}
These well known facts follow easily from the definitions.
See for instance \cite[Theorem IV.3.1]{Mac} for details.
\end{proof}

Luckily for us, in triangulated categories,
all monomorphisms and epimorphisms split.
This is also well known, but it seems hard to find a citable reference,
so we include the easy proof (of a corollary of this fact) for the benefit of the reader.

\begin{prop}\label{prop:split}
Let $\mathcal{T}$ be a triangulated category.
\begin{enumerate}
\item If $f:X \to Y$ is a monomorphism in $\mathcal{T}$,
there is an object $Z$ in $\mathcal{T}$ such that $Y \cong X \oplus Z$.
\item If $f:X \to Y$ is an epimorphism in $\mathcal{T}$,
there is an object $Z$ in $\mathcal{T}$ such that $X \cong Y \oplus Z$.
\end{enumerate}
\end{prop}
\begin{proof}
Assuming that $f$ is a monomorphism,
embed $f$ in a distinguished triangle of the form
\[
Z \xrightarrow{g} X \xrightarrow{f} Y \to Z[1]
\]
As shown in \cite[Proposition 1.1(a)]{RD}, the axioms of a triangulated category imply that $f\circ g = 0$. 
Since $f$ is a monomorphism, 
this implies that $g = 0$.
Hence, by \cite[tag 05QT]{SP},
this implies that $Y \cong X \oplus Z[1]$.
The proof of the second statement is almost identical,
so we omit it.
\end{proof}

Two classes of DG-modules that will play an important role in the sequel are the derived injective and derived projective DG-modules.
These generalize to the DG setting the injective and projective modules over an ordinary ring.
The derived injective DG-modules, 
first introduced in \cite{ShINJ},
and studied also in \cite{Min},
are the set $\opn{Inj}(A)$ of left DG-modules $I$,
with the property that either $I \cong 0$,
or $\injdim_A(I) = 0 = \inf(I)$,
where $\inf(I) = \inf\{n\in \mathbb{Z} \mid \mrm{H}^n(I) \ne 0\}$.
Dually, the derived projective DG-modules,
introduced in \cite{Min},
are the set $\opn{Proj}(A)$ of DG-modules $P$ such that either $P \cong 0$,
or $\projdim_A(P) = 0 = \sup(P)$,
with $\sup(P) = \sup\{n\in \mathbb{Z} \mid \mrm{H}^n(P) \ne 0\}$. 
Alternatively, objects of $\opn{Proj}(A)$ are exactly the direct summands of direct sums of copies of $A$.

\section{The reduction and the coreduction functors}

In this section we make a detailed study of the reduction functor 
\[
F(-):=\mrm{H}^0(A)\otimes^{\mrm{L}}_A -:\cat{D}(A) \to \cat{D}(\mrm{H}^0(A)),
\]
and the coreduction functor
\[
G(-):=\mrm{R}\opn{Hom}_A(\mrm{H}^0(A),-) : \cat{D}(A) \to \cat{D}(\mrm{H}^0(A)).
\]

We first show that the coreduction functor is never full and never faithful.

\begin{thm}\label{thm:coreduction}
Let $A$ be a non-positive DG-ring which is not equivalent to a ring.
Then the coreduction functor $G:\cat{D}(A) \to \cat{D}(\mrm{H}^0(A))$,
given by $G(-):=\mrm{R}\opn{Hom}_A(\mrm{H}^0(A),-)$ is not faithful and is not full.
\end{thm}
\begin{proof}
Let $S:\cat{D}(\mrm{H}^0(A)) \to \cat{D}(A)$ be the forgetful functor, so that $G$ is right adjoint to $S$.
Since $A$ is not equivalent to a ring,
there exist some $n<0$ such that $\mrm{H}^{n}(A) \ne 0$.
Considering $\mrm{H}^n(A)$ as a left $\mrm{H}^0(A)$-module,
let $\bar{E}$ be a left injective $\mrm{H}^0(A)$-module such that there is a monomorphism $g:\mrm{H}^n(A) \inj \bar{E}$.
By \cite[Theorem 5.7]{ShINJ},
there exist $E \in \opn{Inj}(A)$ such that $\mrm{H}^0(E) = \bar{E}$.
To show that $G$ is not full and not faithful,
since $G$ is a right adjoint,
by \cref{prop:adjfaith},
it is enough to show that the counit map
\[
S(G(E)) \to E
\]
is not an epimorphism and is not a split monomorphism.
By \cite[Proposition 3.8]{ShINJ},
there is an isomorphism
\[
G(E) = \mrm{R}\opn{Hom}_A(\mrm{H}^0(A),E) \cong \mrm{H}^0(E) = \bar{E}.
\]
Assuming there is an epimorphism $\mrm{H}^0(E) \to E$,
it follows from \cref{prop:split} that there is some $Z \in \cat{D}(A)$ such that $\mrm{H}^0(E)\cong E \oplus Z$.
We will show this cannot be the case.
First, note that
\[
\mrm{H}^{-n}\left(\mrm{H}^0(E)\right) = 0,
\]
while
\[
\mrm{H}^{-n}(E\oplus Z) = \mrm{H}^{-n}(E)\oplus \mrm{H}^{-n}(Z).
\]
By \cite[Corollary 4.12]{ShINJ},
there is an isomorphism
\[
\mrm{H}^{-n}(E) \cong \opn{Hom}_{\mrm{H}^0(A)}(\mrm{H}^n(A),\mrm{H}^0(E)).
\]
Since $0\ne g \in \opn{Hom}_{\mrm{H}^0(A)}(\mrm{H}^n(A),\mrm{H}^0(E))$,
we see that $\mrm{H}^{-n}(E) \ne 0$,
showing that there is no epimorphism $S(G(E)) \to E$.
Supposing that there is a split monomorphism $\mrm{H}^0(E) \to E$, it follows that $\mrm{H}^0(E)$ is a direct summand of $E$, so by \cite[Proposition 5.2]{ShINJ} we deduce that $\mrm{H}^0(E) \in \opn{Inj}(A)$.
But then, it follows from \cite[Corollary 4.12]{ShINJ} that
\[
0 = \mrm{H}^{-n}\left(\mrm{H}^0(E)\right) =
\opn{Hom}_{\mrm{H}^0(A)}(\mrm{H}^n(A),\mrm{H}^0(E)) \ne 0,
\]
which is a contradiction.
Since $S(G(E)) \to E$ is not an epimorphism and not a split monomorphism, 
we deduce that $G$ it not faithful and is not full.
\end{proof}

We now prove the dual result about the reduction functor.
The proof is similar, but easier, because $A$ itself is derived projective.
\begin{thm}
Let $A$ be a non-positive DG-ring which is not equivalent to a ring.
Then the reduction functor $F:\cat{D}(A) \to \cat{D}(\mrm{H}^0(A))$,
given by $F(-):=\mrm{H}^0(A)\otimes^{\mrm{L}}_A -$ is not faithful and is not full.
\end{thm}
\begin{proof}
As in the proof of \cref{thm:coreduction},
the functor $F$ is left adjoint to the forgetful functor $S:\cat{D}(\mrm{H}^0(A)) \to \cat{D}(A)$.
Let $n<0$ be such that $\mrm{H}^n(A) \ne 0$.
Applying \cref{prop:adjfaith} for $M = A$
we must show that the map
\[
A \to S(\mrm{H}^0(A)\otimes^{\mrm{L}}_A A)
\]
is not a monomorphism and is not a split epimorphism.
Clearly, $\mrm{H}^0(A)\otimes^{\mrm{L}}_A A \cong \mrm{H}^0(A)$.
If there was a monomorphism $A \to \mrm{H}^0(A)$,
by \cref{prop:split} we would deduce that $A$ is a direct summand of $\mrm{H}^0(A)$,
and then the fact that $\mrm{H}^n(A) \ne 0$ would imply that $\mrm{H}^n(\mrm{H}^0(A)) \ne 0$, which is absurd.
On the other hand, if there was a split epimorphism $A \to \mrm{H}^0(A)$,
the by \cref{prop:split}, 
the DG-module $\mrm{H}^0(A)$ would have to be a direct summand of $A$,
which would imply that $\mrm{H}^0(A) \in \opn{Proj}(A)$.
Since $\mrm{H}^0(A) = \mrm{H}^0(\mrm{H}^0(A))$,
this would imply by \cite[Lemma 2.8(3)]{Min} that there is an isomorphism $A \cong \mrm{H}^0(A)$,
which contradicts the assumption that $A$ is an honest DG-ring.
Hence, $F$ is not faithful and it not full.
\end{proof}

We finish this section by discussing the conservative property for the reduction and coreduction functors.
Recall, as discussed in the introduction,
that the reduction functor is always conservative on the bounded above derived category,
and that the coreduction functor is always conservative on the bounded below derived category.
Moreover, if $A$ itself has bounded cohomology, 
then reduction and coreduction are both conservative on the entire unbounded derived category.
Thus, the only remaining question is whether they remain conservative on the unbounded derived category if $A$ has unbounded cohomology. The next result shows that in general, they are not. 

\begin{thm}
There exist a commutative DG-ring $A$ with unbounded cohomology such that the reduction functor $\mrm{H}^0(A)\otimes^{\mrm{L}}_A -:\cat{D}(A) \to \cat{D}(\mrm{H}^0(A))$ and the coreduction functor $\mrm{R}\opn{Hom}_A(\mrm{H}^0(A),-):\cat{D}(A) \to \cat{D}(\mrm{H}^0(A))$ are not conservative.
\end{thm}
\begin{proof}
Let $\K$ be a field,
and let $A=\K[t]$ be a polynomial ring over $\K$,
considered as a non-positive DG-ring with zero differential,
where $\deg(t) = -2$.
This DG-ring was also considered in \cite[Example 7.26]{YeDual},
where it was observed that there is a map of DG-modules,
the inclusion map $\phi:A[2] \to A$
(which one can consider simply as multiplication by $t$),
and that moreover the cone of $\phi$ is naturally isomorphic to $\mrm{H}^0(A)$.
In other words,
there is a distinguished triangle
\begin{equation}\label{eqn:triangle}
A[2] \xrightarrow{\phi} A \to \mrm{H}^0(A) \to \left(A[2]\right)[1]
\end{equation}
in $\cat{D}(A)$.
Consider the DG-module $M = \K[t,t^{-1}]$.
Here, $t$ is of degree $-2$, while $t^{-1}$ is of degree $+2$.
Thus, as a graded abelian group $M^{2n} = \K$ and $M^{2n+1} = 0$ for all $n \in \mathbb{Z}$,
and $M$ becomes a DG-module over $A$ in the obvious way with the zero differential. 
Since $M$ has non-zero cohomology,
it follows that $M \ncong 0$ in $\cat{D}(A)$.
To calculate its reduction $\mrm{H}^0(A)\otimes^{\mrm{L}}_A M$,
we apply the triangulated functor $-\otimes^{\mrm{L}}_A M$ to \cref{eqn:triangle},
and obtain the distinguished triangle
\[
A[2] \otimes^{\mrm{L}}_A M \to A \otimes^{\mrm{L}}_A M \to \mrm{H}^0(A) \otimes^{\mrm{L}}_A M \to \left(A[2]\right)[1]\otimes^{\mrm{L}}_A M
\]
Here, the first map is also induced from the multiplication by $t$ map.
However, the definition of $M$ makes it clear that this map is an isomorphism,
which implies that $\mrm{H}^0(A) \otimes^{\mrm{L}}_A M \cong 0$.
This shows that the reduction functor $\mrm{H}^0(A) \otimes^{\mrm{L}}_A - :\cat{D}(A) \to \cat{D}(\mrm{H}^0(A))$ is not conservative.
Similarly,
applying the contravariant triangulated functor $\mrm{R}\opn{Hom}_A(-,M)$ to \cref{eqn:triangle},
we obtain the distinguished triangle
\[
\mrm{R}\opn{Hom}_A(\mrm{H}^0(A),M) \to \mrm{R}\opn{Hom}_A(A,M) \to \mrm{R}\opn{Hom}_A(A[2],M) \to \left(\mrm{R}\opn{Hom}_A(\mrm{H}^0(A),M)\right)[1],
\]
or more explicitly:
\[
\mrm{R}\opn{Hom}_A(\mrm{H}^0(A),M) \to M \to M[-2] \to \left(\mrm{R}\opn{Hom}_A(\mrm{H}^0(A),M)\right)[1].
\]
As above, the map $M \to M[-2]$ is an isomorphism,
which implies that $\mrm{R}\opn{Hom}_A(\mrm{H}^0(A),M) \cong 0$.
Hence, the coreduction functor $\mrm{R}\opn{Hom}_A(\mrm{H}^0(A),-)$ is not conservative.
\end{proof}

\section{The forgetful functor}

In this section discuss the categorical properties of the forgetful functor $S:\cat{D}(\mrm{H}^0(A)) \to \cat{D}(A)$.
In the proof of the next result, 
we can use either derived projectives or derived injectives.
We take derived projectives, 
as they are simpler.
\begin{thm}
Let $A$ be a non-positive DG-ring which is not equivalent to a ring.
Then the forgetful functor $S:\cat{D}(\mrm{H}^0(A))) \to \cat{D}(A)$ is not full.
\end{thm}
\begin{proof}
Since $S$ is right adjoint to $F$,
by \cref{prop:adjfaith},
it is full if and only if 
for any object $M \in \cat{D}(\mrm{H}^0(A))$,
the map $\mrm{H}^0(A)\otimes^{\mrm{L}}_A S(M) \to M$ is a split monomorphism.
If this is the case, 
taking $M = \mrm{H}^0(A)$,
we deduce that $\mrm{H}^0(A)\otimes^{\mrm{L}}_A \mrm{H}^0(A)$ is a direct summand of $\mrm{H}^0(A)$.
Hence, $\mrm{H}^0(A)\otimes^{\mrm{L}}_A \mrm{H}^0(A)$
is isomorphic to a projective $\mrm{H}^0(A)$-module.
By the {K}\"{u}nneth trick (as in the proof of \cite[Proposition 3.1]{YeDual}),
it holds that
\[
\mrm{H}^0\left(\mrm{H}^0(A)\otimes^{\mrm{L}}_A \mrm{H}^0(A)\right) \cong \mrm{H}^0(A),
\]
so the above implies that
\[
\mrm{H}^0(A)\otimes^{\mrm{L}}_A \mrm{H}^0(A) \cong \mrm{H}^0(A).
\]
This in turn implies,
by \cite[Proposition 3.3(1)]{YeDual},
that there is an isomorphism $A \cong \mrm{H}^0(A)$ in $\cat{D}(A)$,
contradicting the fact that $A$ is an honest DG-ring.
Hence, $S$ is not full.
\end{proof}

The careful reader probably noticed that we left faithfulness out of the above result.
The reason for this is that the forgetful functor may be faithful sometimes.

\begin{thm}
There exist a commutative noetherian DG-ring $A$ with bounded cohomology, which is not equivalent to a ring,
such that the forgetful functor $S:\cat{D}(\mrm{H}^0(A)) \to \cat{D}(A)$ is faithful.
\end{thm}
\begin{proof}
Let $A$ be a commutative ring,
and let $M$ be a non-zero $A$-module.
Consider the trivial extension DG-ring $B = A \skewtimes M[1]$.
This is a non-positive DG-ring with zero differential,
with $B^{-1} = M$, $B^0 = A$, and $B^n = 0$ for all $n \notin \{-1,0\}$.
Observe that the natural map $\pi:B \to \mrm{H}^0(B) = A$ has a one-sided inverse $\tau:A \to B$ given by $\tau(a) = a$.
It then holds that $\pi \circ \tau = 1_A$.
If we denote by $\pi_*$ (respectively $\tau_*$) the forgetful functor $\cat{D}(\mrm{H}^0(B)) \to \cat{D}(B)$
(resp. $\cat{D}(B) \to \cat{D}(\mrm{H}^0(B))$),
then it follows that
\[
\tau_*\circ \pi_* = 1_{\cat{D}(A)},
\]
and this implies that the forgetful functor
$\cat{D}(\mrm{H}^0(B)) \to \cat{D}(B)$ is faithful.
\end{proof}

In our next result we show that the forgetful functor $S:\cat{D}(A) \to \cat{D}(\mrm{H}^0(A))$ may fail to be faithful. Constructing such an example is more difficult.

\begin{thm}
There exist a commutative noetherian DG-ring $B$ with bounded cohomology such that the forgetful functor
\[
S:\cat{D}(\mrm{H}^0(B)) \to \cat{D}(B)
\]
is not faithful.
\end{thm}
\begin{proof}
Let $\K$ be a field, and let $A = \K[x,y]/(x\cdot y)$.
We define the DG-ring $B = K(A;x^2)$,
the Koszul complex over $A$ with respect to the element $x^2 \in A$.
Notice that since $x^2$ is not an $A$-regular element,
it holds that $\mrm{H}^{-1}(B) \ne 0$,
so that $B$ is an honest DG-ring which is not equivalent to a ring. We will show that the forgetful functor $S:\cat{D}(\mrm{H}^0(B)) \to \cat{D}(B)$ is not faithful.
Let $M = A/(y) \cong \K[x]$.
We our going to show that the forgetful map
\[
\opn{Hom}_{\cat{D}(\mrm{H}^0(B))}(\mrm{H}^0(M\otimes^{\mrm{L}}_A B),\mrm{H}^0(M\otimes^{\mrm{L}}_A B)[5]) \to 
\opn{Hom}_{\cat{D}(B)}(\mrm{H}^0(M\otimes^{\mrm{L}}_A B),\mrm{H}^0(M\otimes^{\mrm{L}}_A B)[5])
\]
is not injective.
A projective resolution of $M$ is given by the cochain complex
\[
P = \left(\dots \to A \xrightarrow{\cdot x} A \xrightarrow{\cdot y} A \xrightarrow{\cdot x} A \xrightarrow{\cdot y} A \to 0\right)
\]
which is concentrated in cohomological degrees $\le 0$.
Using this projective resolution, 
we see that
\[
\mrm{R}\opn{Hom}_A(M,M) \cong \opn{Hom}_A(P,M) \cong Q
\]
where 
\[
Q = \left(\dots \to 0 \to  M \xrightarrow{0} M \xrightarrow{\cdot x} M \xrightarrow{0} M \xrightarrow{\cdot x} M \to \dots \right)
\]
with $Q$ concentrated in cohomological degrees $\ge 0$.
To proceed with the computation, 
it is helpful to observe that the map
\[
\xymatrix{
0 \ar[r] & M \ar[r]^{0}\ar[d] & M \ar[r]^{\cdot x}\ar[d] & M \ar[r]^0\ar[d] & M \ar[r]^{\cdot x}\ar[d] & M\ar[d]\ar[r] & \dots\\
0 \ar[r] & M \ar[r] & 0 \ar[r] & M/(x\cdot M) \ar[r] & 0 \ar[r] & M/(x\cdot M) \ar[r] & \dots
}
\]
is a quasi-isomorphism.
It follows that there is an isomorphism
\begin{equation}\label{eqn:rhom}
\mrm{R}\opn{Hom}_A(M,M) \cong M \oplus \bigoplus_{n=1}^{\infty} \K[-2\cdot n]
\end{equation}
in $\cat{D}(A)$.
We now use the above to make a computation over the DG-ring $B$. 
The hom-tensor adjunction implies that
\[
\mrm{R}\opn{Hom}_B(M\otimes^{\mrm{L}}_A B,M\otimes^{\mrm{L}}_A B) \cong \mrm{R}\opn{Hom}_A(M,M\otimes^{\mrm{L}}_A B),
\]
and since $B$ is compact over $A$,
by \cite[Theorem 12.9.10]{YeBook},
the derived tensor evaluation map gives an isomorphism
\[
\mrm{R}\opn{Hom}_A(M,M\otimes^{\mrm{L}}_A B) \cong \mrm{R}\opn{Hom}_A(M,M) \otimes^{\mrm{L}}_A B.
\]
Combining this with \cref{eqn:rhom} we see that
\[
\mrm{R}\opn{Hom}_B(M\otimes^{\mrm{L}}_A B,M\otimes^{\mrm{L}}_A B) \cong \left( M \oplus \bigoplus_{n=1}^{\infty} \K[-2\cdot n]\right) \otimes^{\mrm{L}}_A B.
\]
Since derived tensor products commute with direct sums,
we can compute the latter as follows.
First, we note that
\[
M \otimes^{\mrm{L}}_A B \cong \left(\dots 0 \to M \xrightarrow{\cdot x^2} M \to 0 \to \dots \right) \cong \K[x]/(x^2).
\]
Secondly, we have that 
\[
\K \otimes^{\mrm{L}}_A B \cong \left(\dots 0 \to \K \xrightarrow{\cdot x^2} \K \to 0 \to \dots \right) \cong \K \oplus \K[1].
\]
Hence, we obtain
\[
\mrm{R}\opn{Hom}_B(M\otimes^{\mrm{L}}_A B,M\otimes^{\mrm{L}}_A B) \cong \K[x]/(x^2) \oplus \bigoplus_{n=1}^{\infty} \K[-n].
\]
This shows that for all $n \ge 1$ it holds that
\[
\opn{Hom}_{\cat{D}(B)}(M\otimes^{\mrm{L}}_A B,M\otimes^{\mrm{L}}_A B[n]) = 
\opn{Ext}^n_B(M\otimes^{\mrm{L}}_A B,M\otimes^{\mrm{L}}_A B) \cong \K.
\]
We have seen above that 
\[
M \otimes^{\mrm{L}}_A B \cong \K[x]/(x^2) \cong \mrm{H}^0(M \otimes^{\mrm{L}}_A B).
\]
Thus, the DG-module $M \otimes^{\mrm{L}}_A B$ is in the image of the forgetful functor $S:\cat{D}(\mrm{H}^0(B)) \to \cat{D}(B)$. 
This is simply because for the $A$-module $M$, 
it holds that $M$ is $x^2$-regular,
so that $M \otimes^{\mrm{L}}_A B \cong M/x^2M$.
In other words, 
it holds that
\[
\mrm{R}\opn{Hom}_B(M\otimes^{\mrm{L}}_A B,M\otimes^{\mrm{L}}_A B) = \mrm{R}\opn{Hom}_B(S(M/x^2M),S(M/x^2M)).
\]
Our next task thus is to compute
\[
\opn{Ext}^n_{\mrm{H}^0(B)}(M/x^2M,M/x^2M).
\]
By definition, we have that $\mrm{H}^0(B) = \K[x,y]/(x\cdot y,x^2)$. 
Let us denote this ring by $C$.
Computing $\opn{Ext}^n_C(M/x^2M,M/x^2M) = \opn{Ext}^n_C(C/y,C/y)$ is not difficult,
but is a tedious task. 
To avoid these computations, 
we use the computer algebra system Macaulay2 \cite{M2}.
Using it, one may verify that
\[
\opn{Ext}^5_C(C/y,C/y) \cong C^2/\opn{Im}(\psi)
\]
where $\psi:C^4 \to C^2$ is given by 
\[
\psi(a,b,c,d) = (a \cdot x + b\cdot y, c\cdot x + d\cdot y).
\]
It follows that 
\[
\dim_{\K}\left(\opn{Ext}^5_{\mrm{H}^0(B)}(M/x^2M,M/x^2M)\right) = 2.
\]
But we have seen above that
\[
\dim_{\K}\left(\opn{Ext}^5_B(S(M/x^2M),S(M/x^2M))\right) = 1,
\]
so we conclude that the map
\[
S:\opn{Hom}_{\cat{D}(\mrm{H}^0(B))}(M/x^2M,M/x^2M[5]) \to 
\opn{Hom}_{\cat{D}(B)}(S(M/x^2M),S(M/x^2M[5]))
\]
cannot be injective. Hence, the functor $S$ is not faithful.
\end{proof}

\begin{rem}
The reader might wonder how did we arrive to this example.
Our basic idea was to find a DG-ring $B$ which has a rather nice homological behavior, 
but such that $\mrm{H}^0(B)$ has bad homological behavior.
This works because rings with nice homological behavior tend to have smaller $\opn{Ext}$ modules.
The point in the above example is that the ring $A$ we started with is a Cohen-Macaulay ring,
and hence, by \cite[Corollary 4.6]{ShKoszul},
the DG-ring $B$, being a Koszul DG-ring over a Cohen-Macaulay ring, is a Cohen-Macaulay DG-ring.
On the other hand, the ring
\[
\mrm{H}^0(B) = \K[x,y]/(x\cdot y,x^2)
\]
is not Cohen-Macaulay, 
and this allowed the above bad behavior to happen.
\end{rem}

Despite the failure of the forgetful functor to be faithful,
we next show that under commutativity and boundedness assumptions, we have the following simultaneously vanishing results which the forgetful functor can detect:

\begin{thm}\label{thm:sym-vanish}
Let $A$ be a commutative DG-ring,
and suppose that $A$ has bounded cohomology and the commutative ring $\mrm{H}^0(A)$ is noetherian.
Let $S:\cat{D}(\mrm{H}^0(A)) \to \cat{D}(A)$ be the forgetful functor,
and let $M,N \in \cat{D}(\mrm{H}^0(A))$.
Then the following holds:
\begin{enumerate}
\item $S(M)\otimes^{\mrm{L}}_A S(N) \cong 0$ if and only if  $M\otimes^{\mrm{L}}_{\mrm{H}^0(A)} N \cong 0$.
\item $\mrm{R}\opn{Hom}_A(S(M),S(N)) \cong 0$ if and only if $\mrm{R}\opn{Hom}_{\mrm{H}^0(A)}(M,N) \cong 0$
\end{enumerate}
\end{thm}

We require some preliminaries before proving \cref{thm:sym-vanish}.
If $A$ is a commutative non-positive DG-ring,
then $\cat{D}(A)$ is equipped with a natural action of the ring $\mrm{H}^0(A)$. This allows one to apply the Benson-Iyengar-Krause framework of support and cosupport \cite{BIK,BIK2}. 
In particular, associated to any prime ideal $\bar{\p} \in \opn{Spec}(\mrm{H}^0(A))$ are functors $\Gamma_{\bar{\p}}$,
$L_{\bar{\p}}$, $\Lambda_{\bar{\p}}$ and  $V_{\bar{\p}}$,
called the local cohomology, localization, completion and colocalization functors at $\bar{\p}$. 
When $A$ is an ordinary commutative noetherian ring,
these coincide with the usual local cohomology, localization, derived completion and colocalization functors.
These are all triangulated functors $\cat{D}(A) \to \cat{D}(A)$.
In this differential graded setting,
they were also studied in detail in \cite{ShComp,ShWi}.
Using these functors, one defines the (big) support and the cosupport of a DG-module $M$ by the formulas
\[
\mrm{supp}_A(M) = \{\bar{\p} \in \opn{Spec}(\mrm{H}^0(A)) \mid \Gamma_{\bar{\p}} L_{\bar{\p}}(M) \ncong 0\}
\]
and
\[
\mrm{cosupp}_A(M) = \{\bar{\p} \in \opn{Spec}(\mrm{H}^0(A)) \mid \Lambda_{\bar{\p}} V_{\bar{\p}}(M) \ncong 0\}.
\]

\begin{lem}\label{lem:support}
Let $A$ be a commutative DG-ring,
with a forgetful functor $S:\cat{D}(\mrm{H}^0(A)) \to \cat{D}(A)$.
Then for any $M \in \cat{D}(\mrm{H}^0(A))$ there are equalities
\[
\mrm{supp}_{\mrm{H}^0(A)}(M) = \mrm{supp}_A(S(M))
\]
and 
\[
\mrm{cosupp}_{\mrm{H}^0(A)}(M) = \mrm{cosupp}_A(S(M)).
\]
\end{lem}
\begin{proof}
According to \cite[Theorem 7.7]{BIK2},
the functor $S$ commutes with $\Gamma_{\bar{\p}}$,
$L_{\bar{\p}}$, $\Lambda_{\bar{\p}}$ and  $V_{\bar{\p}}$.
The result then follows from the fact that $S$ is conservative.
\end{proof}

We now use the above to prove \cref{thm:sym-vanish}.
\begin{proof}[Proof of Theorem \ref{thm:sym-vanish}]
The assumptions on $A$ imply by \cite[Corollary 4.12]{ShWi} that $\cat{D}(A)$ is stratified and costratified by the canonical action of $\mrm{H}^0(A)$.
Similarly, the Hopkins-Neeman theorem \cite{Hopkins,Neeman1992} about stratification and the Neeman theorem \cite{Neeman11} about costratification,
imply that $\cat{D}(\mrm{H}^0(A))$ is also stratified and costratified by its canonical $\mrm{H}^0(A)$ action.
Given $M,N \in \cat{D}(\mrm{H}^0(A))$,
it follows that $S(M)\otimes^{\mrm{L}}_A S(N) \cong 0$ if and only if $\mrm{supp}_A(S(M)\otimes^{\mrm{L}}_A S(N)) = \emptyset$,
and that $\mrm{R}\opn{Hom}_A(S(M),S(N)) \cong 0$
if and only if $\mrm{cosupp}_A(\mrm{R}\opn{Hom}_A(S(M),S(N))) = \emptyset$.
Similarly, we have that 
$M\otimes^{\mrm{L}}_{\mrm{H}^0(A)} N \cong 0$ 
if and only if $\mrm{supp}_{\mrm{H}^0(A)}(M\otimes^{\mrm{L}}_{\mrm{H}^0(A)} N) = \emptyset$, and that
$\mrm{R}\opn{Hom}_{\mrm{H}^0(A)}(M,N) \cong 0$
if and only if 
 $\mrm{cosupp}_{\mrm{H}^0(A)}(\mrm{R}\opn{Hom}_{\mrm{H}^0(A)}(M,N)) = \emptyset$.
According to \cite[Theorem 4.15(3)]{ShWi}, 
there is an equality
\[
\mrm{supp}_A(S(M)\otimes^{\mrm{L}}_A S(N)) = \mrm{supp}_A(S(M)) \cap \mrm{supp}_A(S(N)),
\]
while by \cite[Theorem 4.17(4)]{ShWi},
we have that
\[
\mrm{cosupp}_A(\mrm{R}\opn{Hom}_A(S(M),S(N))) = \mrm{supp}_A(S(M)) \cap \mrm{cosupp}_A(S(N)).
\]
But as $\cat{D}(\mrm{H}^0(A))$ is also stratified and costratified by the canonical action of $\mrm{H}^0(A)$,
we also have equalities
\[
\mrm{supp}_{\mrm{H}^0(A)}(M\otimes^{\mrm{L}}_{\mrm{H}^0(A)} N) = \mrm{supp}_{\mrm{H}^0(A)}(M) \cap \mrm{supp}_{\mrm{H}^0(A)}(N)
\]
and
\[
\mrm{cosupp}_{\mrm{H}^0(A)}(\mrm{R}\opn{Hom}_{\mrm{H}^0(A)}(M,N)) = \mrm{supp}_{\mrm{H}^0(A)}(M) \cap \mrm{cosupp}_{\mrm{H}^0(A)}(N).
\]
Combining these facts with \cref{lem:support},
we deduce that there are equalities
\[
\mrm{supp}_A(S(M)\otimes^{\mrm{L}}_A S(N)) = 
\mrm{supp}_{\mrm{H}^0(A)}(M\otimes^{\mrm{L}}_{\mrm{H}^0(A)} N)
\]
and
\[
\mrm{cosupp}_A(\mrm{R}\opn{Hom}_A(S(M),S(N))) = \mrm{cosupp}_{\mrm{H}^0(A)}(\mrm{R}\opn{Hom}_{\mrm{H}^0(A)}(M,N)).
\]
This implies the result.
\end{proof}

In the next corollary, note that \textbf{we do not} assume that $\mathbf{x}$ is an $A$-regular sequence.

\begin{cor}\label{cor:torExt}
Let $A$ be a commutative noetherian ring,
let $\mathbf{x}=x_1,\dots,x_n$ be a finite sequence of elements in $A$, and let $M,N$ be $A$-modules.
Assume that $\mathbf{x}$ is both an $M$-regular and an $N$-regular sequence.
\begin{enumerate}
\item If $\opn{Tor}_n^A(M,N) = 0$ for all $n\ge 0$,
then $\opn{Tor}_n^{A/\mathbf{x}A}(M/\mathbf{x}M,N/\mathbf{x}N) = 0$ for all $n \ge 0$.
\item If $\opn{Ext}^n_A(M,N) = 0$ for all $n\ge 0$,
then $\opn{Ext}^n_{A/\mathbf{x}A}(M/\mathbf{x}M,N/\mathbf{x}N) = 0$ for all $n \ge 0$.
\end{enumerate}
\end{cor}
\begin{proof}
Let $B = K(A;\mathbf{x})$ be the Koszul complex of $A$ with respect to $\mathbf{x}$, considered as a DG-ring. 
Note that $B$ is a commutative noetherian DG-ring with bounded cohomology such that $\mrm{H}^0(B) = A/\mathbf{x}A$.
The assumptions that $\opn{Tor}_n^A(M,N) = 0$ for all $n\ge 0$, or that $\opn{Ext}^n_A(M,N) = 0$ for all $n\ge 0$ imply that $M\otimes^{\mrm{L}}_A N \cong 0$,
or that $\mrm{R}\opn{Hom}_A(M,N) \cong 0$.
In the former case, by associativity of the derived tensor product, this implies that
\[
(M\otimes^{\mrm{L}}_A B) \otimes^{\mrm{L}}_B (N\otimes^{\mrm{L}}_A B) \cong (M\otimes^{\mrm{L}}_A N) \otimes^{\mrm{L}}_A B \cong 0,
\]
while in the latter case, 
the fact that $B$ is compact over $A$ implies that
\[
\mrm{R}\opn{Hom}_B(M\otimes^{\mrm{L}}_A B, N\otimes^{\mrm{L}}_A B) \cong 
\mrm{R}\opn{Hom}_A(M,N) \otimes^{\mrm{L}}_A B \cong 0.
\]
The assumption that $\mathbf{x}$ is $M$-regular and $N$-regular implies that
\[
M\otimes^{\mrm{L}}_A B \cong \mrm{H}^0(M\otimes^{\mrm{L}}_A B) = M/\mathbf{x}M
\]
and
\[
N\otimes^{\mrm{L}}_A B \cong \mrm{H}^0(N\otimes^{\mrm{L}}_A B) = N/\mathbf{x}N.
\]
Thus, if $S:\cat{D}(\mrm{H}^0(B)) \to \cat{D}(B)$ is the forgetful functor, 
applying \cref{thm:sym-vanish}, 
we see that
\[
(M\otimes^{\mrm{L}}_A B) \otimes^{\mrm{L}}_B (N\otimes^{\mrm{L}}_A B) =
S(M/\mathbf{x}M) \otimes^{\mrm{L}}_B S(M/\mathbf{x}M) \cong 0
\]
if and only if 
\[
(M/\mathbf{x}M) \otimes^{\mrm{L}}_{\mrm{H}^0(B)} (N/\mathbf{x}N) \cong 0.
\]
Similarly, we deduce that
\[
\mrm{R}\opn{Hom}_B(M\otimes^{\mrm{L}}_A B, N\otimes^{\mrm{L}}_A B) = \mrm{R}\opn{Hom}_B(S(M/\mathbf{x}M),S(N/\mathbf{x}N)) \cong 0
\]
if and only if 
\[
\mrm{R}\opn{Hom}_{\mrm{H}^0(B)}(M/\mathbf{x}M,N/\mathbf{x}N) \cong 0.
\]
The result follows from the observation that
$\opn{Tor}_n^{A/\mathbf{x}A}(M/\mathbf{x}M,N/\mathbf{x}N) = 0$ for all $n \ge 0$ is equivalent to 
\[
(M/\mathbf{x}M) \otimes^{\mrm{L}}_{\mrm{H}^0(B)} (N/\mathbf{x}N) \cong 0,
\]
and that $\opn{Ext}^n_{A/\mathbf{x}A}(M/\mathbf{x}M,N/\mathbf{x}N) = 0$ for all $n \ge 0$ is equivalent to 
\[
\mrm{R}\opn{Hom}_{\mrm{H}^0(B)}(M/\mathbf{x}M,N/\mathbf{x}N) \cong 0.
\]
\end{proof}

\begin{rem}
In the terminology of \cite{KS},
the above corollary says that if $M,N$ are a pair of Ext-orthogonal modules over a commutative noetherian ring $A$, 
and if $\mathbf{x}$ is both $M$-regular and $N$-regular,
then the pair $M/\mathbf{x}M$ and $N/\mathbf{x}N$ is Ext-orthogonal over $A/\mathbf{x}$.
\end{rem}

\textbf{Acknowledgments.}

The author is thankful to an anonymous referee for suggestions that helped improving this manuscript. This work has been supported by the grant GA~\v{C}R 20-02760Y from the Czech Science Foundation.

\bibliographystyle{abbrv}
\bibliography{main}

\end{document}